\documentclass[12pt, reqno]{article}
\usepackage{latexsym}
\usepackage{color, hyperref, amsmath, amsthm, breqn}
\usepackage{amssymb}
\usepackage{amsfonts}

\newcommand{\anu}{\begin{enumerate}}
\newcommand{\pab}{\end{enumerate}}
\newcommand{\ee}{\end{equation}} 
\newcommand{\bb}{\begin{equation}}

\newtheorem{lemma}{Lemma}[section]
\newtheorem{rem}{Remark}[section]
\newtheorem{theorem}{Theorem}[section]
\newtheorem{cor}{Corollary}[section]

\title{ New upper bounds for Ramanujan primes 
}
\author{
Anitha Srinivasan\thanks{Saint Louis University-
Madrid Campus,
Avenida del Valle 34,
28003 Madrid, Spain.
email: rsrinivasan.anitha{@}gmail.com }, 
Pablo Ar\'es\thanks{Universidad San Pablo CEU, 
Juli\'an Romea, 23, 28003 Madrid, Spain.
email: pablo.aresgastesi{@}ceu.es}
}
\date{}
\begin{document}

\maketitle

\setcounter{page}{1}

 \begin{abstract} 

For $n\ge 1$, the $n^{\rm th}$ Ramanujan prime is defined as the smallest 
 positive integer $R_n$ such that for all $x\ge R_n$, the interval
$(\frac{x}{2}, x]$ has at least $n$ primes. 
We show that for every $\epsilon>0$, there is a positive integer $N$ such that
if $\alpha=2n\left(1+\dfrac{\log 2+\epsilon}{\log n+j(n)}\right)$, then 
 $R_n< p_{[\alpha]}$ for all $n>N$, 
 where $p_i$ is the $i^{\rm th}$ prime and
$j(n)>0$ is any function that satisfies 
$j(n)\to \infty$ and $nj'(n)\to 0$.
 \end{abstract}

 \begin{section}{\bf Introduction}\end{section}
 For $n\ge 1$, the $n^{\rm th}$ Ramanujan prime is defined as the  smallest
 positive integer $R_n$, such that for all $x\ge R_n$, the interval 
$(\frac{x}{2}, x]$ has at least $n$ primes. Note that by the minimality 
condition, $R_n$ is prime and the interval 
$(\frac{R_n}{2}, R_n]$ contains exactly $n$ primes. 
Let $R_n=p_s$,  
where $p_i$ denotes the $i^{\rm th}$ prime. 
 Sondow \cite{So} 
showed that $p_{2n}<R_n<p_{4n}$ for all $n$, and conjectured 
that $R_n<p_{3n}$ for all $n$. This conjecture was proved by 
Laishram \cite{La}, and the upper bound $p_{3n}$ improved by various authors (\cite{Ax1}, \cite{SNN}). 
%who showed 
%that $R_n<\frac{41}{47}p_{3n}$.
Subsequently, Srinivasan \cite{As} and Axler \cite{Ax1}  improved these bounds by showing that 
 for every $\epsilon>0$, there
exists an integer $N$ such that
 $$R_n<p_{[2n(1+\epsilon)]}
 {\text { for all }} n>N.$$ 
 Using the method in \cite{As}
(outlined below), 
 a further improvement was presented by Srinivasan and Nicholson, who proved 
that
$$s< 2n\left(1+\frac{3}{\log n+\log(\log n)-4}\right)$$ for 
all $n>241$. 
The above result follows from a special case of our main theorem given below.  
 Yang and Togbe \cite{YT}, also used the method in \cite{As}, to give tight upper  
and lower bounds for $R_n$ for large $n$ (greater than $10^{300}$).
%$2n\left(1+\frac{\log(2)}{\log n}-f(n)\right)$. 
For some interesting generalizations of Ramanujan primes the reader may refer to 
\cite{Ax}, 
 \cite{Pak} and 
\cite{Sh}.

 The main idea in \cite{As} is to define 
a function $F(x)$ that is decreasing for $x\ge 2n$ and that satisfies
$F(s)>0$. Then, an $\alpha>2n$ is found such that $F(\alpha)<0$ for $n>N$, which would imply 
that $s<\alpha$ for $n>N$ given the decreasing nature of $F$. 
We employ a variation of this method, where we first show that  
$F(\alpha)$ is a decreasing function for $n>N$. Then we find an integer greater 
than $N$ for which $F(\alpha)<0$, which leads us to the desired result. 
Our main result is the following.
\begin{theorem}
Let $R_n=p_s$
and $\epsilon>0$. 
Let $j(n)>0$ be a function such that 
$j(n)\to \infty$ and $nj'(n)\to 0$ 
as $n\to \infty$ and 
let $$g(n)=\frac{\log n +j(n)}{\log 2+\epsilon}.$$ 
Then there exists a positive integer $N$ such that
for all $n>N$, we have $s<\alpha$, where 
$\alpha=2n\left(1+\frac{1}{g(n)}\right)$.

\end{theorem}
In the following corollary we record a bound obtained with $\epsilon=0.5$, where 
$j(n)$ is chosen so as to minimize the number of calculations. 
  Similar results can be given for smaller values of $\epsilon$ (with different $j(n)$) 
where the determination of $N$ depends solely on computational power.
\begin{cor} Let $R_n=p_s$. Then for $n>43$ we have 
$s<
2n\left(1+\frac{1}{g(n)}\right)$, where 
$$g(n)=\frac{\log n +\log_2 n-\log 2-0.5}{\log 2+0.5}.$$ 
\end{cor}
%\begin{rem}
%The result in Corollary 1.1  improves upon previous results that are  
%either valid only for large $n$ ( as in \cite{YT}), 
%or, have an upper bound that is greater than the one given 
%in the corollary ( as in \cite{SN}). 
%\end{rem}
\section{The basic functions and lemmas}
We will use the following bounds for the $k^{\rm th}$ prime given by Dusart. 
\begin{lemma}
The following hold for the $k^{\rm th}$ prime $p_k$.
\begin{enumerate}
\item
$p_k>k\left(\log k+\log_2 k-1+\frac{\log_2 k-2.1}{\log k}\right)$ for all $k\ge 3$.
\item
$p_k<k\left(\log k+\log_2 k-1+\frac{\log_2 k-2}{\log k}\right)$ for all $k\ge 688383$.

\end{enumerate}
\end{lemma}
\begin{proof} See \cite{Du}
\end{proof}
Let $$U(k)=k\left(\log k+\log_2 k-1+\frac{\log_2 k-2}{\log k}\right)$$ and 
$$L(k)=k\left(\log k+\log_2 k-1+\frac{\log_2 k-2.1}{\log k}\right).$$ 
Note that $U(x)=L(x)+f(x)$ where
$f(x)=\dfrac{0.1x}{\log x}$. We define 
$$F(x, n)=U(x)-2L(x-n)=U(x)-2U(x-n)+2f(x-n)$$ 
and 
$$G(n)=F(\alpha, n), $$ 
where $\alpha=2n\left(1+\frac{1}{g(n)}\right)$ and 
$g(n)$ is a function that satisfies $g(n)\ge 1$ and 
$g(n)\to \infty$ as $n\to \infty$. 
\begin{lemma}
Let $R_n=p_s$. Then the following hold.
\begin{enumerate}
\item
$p_{s-n}<\frac{1}{2} p_s.$
\item
$2n<\alpha<2.4n$ for all $n>43$.
\item
 $F(x, n)$ is a decreasing function for all $x\ge 2n$ and 
 $F(s, n)>0$ 
for $n\ge 688383$.
\end{enumerate}
\end{lemma}
\begin{proof}  For parts 1 and 2 see \cite[Lemma 2.1]{As} and \cite[Remark 2.1]{As} respectively. 
For part 3 see \cite{YT}.
$\qed$
\end{proof}

The following lemma contains useful results that include an expression for the derivative $G'(n)$ in terms of the
function $U(x)$.
\begin{lemma}
Let $A=U'(\alpha)-U'(\alpha-n)$. Then the following hold.
\anu
\item $A=A(n)\to \log 2$ as  
$n \to \infty$. 
\item
$
\frac{1}{2}G'(n) 
=A+f'(\alpha-n)+
\left(\frac{n}{g(n)}\right)'
(A-U'(\alpha-n)+2f'(\alpha-n)).
$
\item
$L'(x)>\log x+\log_2 x$ for $x>20$.
\item
$
 A+f'(\alpha-n)-\log 2<
 \log\left( \frac{\log \alpha}{\log(\alpha-n)}\right) +
\frac{\log_2 \alpha}{\log \alpha}+\frac{1.1}{\log(\alpha-n)}+\frac{\log_2(\alpha-n)}{\log^2(\alpha-n)}$.

\pab
\end{lemma}
\begin{proof}
We have 
\begin{equation}
U' (x)=\log x+\log_2 x-\frac{1}{\log x}+\frac{3}{\log^2 x}-\frac{\log_2 x}{\log^2 x}
+\frac{\log_2 x}{\log x}
\end{equation}
and hence
$$A=
\log\left(\frac{\alpha}{\alpha-n}\right)
+
\log\left(\frac{\log(\alpha)}{\log(\alpha-n)}\right)
+t(n),
$$
where $t(n)\to 0$ as $n\to \infty$.  
As $\alpha=2n\left(1+\frac{1}{g(n)}\right)$ and $g(n)\to \infty$, we have 
$A\to\log 2$.

For the second part of the lemma, 
$G(n)=U(\alpha)-2U(\alpha-n)+2f(\alpha-n)$, which gives
$G'(n)=U'(\alpha)\alpha'-2U'(\alpha-n)(\alpha'-1)+2f'(\alpha-n)(\alpha'-1)$. 
As $\alpha'=2+2\left(\frac{n}{g(n)}\right)'$, we have 
$$
\frac{1}{2}G'(n) 
 =U'(\alpha)
\left(1+\left(\frac{n}{g}\right)'\right)+
\left(
1+2\left(\frac{n}{g}\right)'
\right)
(f'(\alpha-n)-U'(\alpha-n))
$$
and the result follows by the definition of $A$.

For part 3 we have 
$$L'(x)=\log x+\log_2 x+\frac{\log_2 x}{\log x}-\frac{\log_2 x}{\log^2 x}
-\frac{1.1}{\log x}+\frac{3.1}{\log^2 x}$$
from which the claim follows as for $n>20$
we have $\frac{\log_2 x}{\log x}-\frac{\log_2 x}{\log^2 x}
-\frac{1.1}{\log x}>0$.

For the last part, we have 
\begin{align*}
&A-\log 2+f'(\alpha-n)  \\
 =&  \log\left( \frac{\log \alpha}{\log(\alpha-n)}\right) +
\frac{\log_2 \alpha}{\log \alpha}+\frac{1.1}{\log(\alpha-n)}+\frac{\log_2(\alpha-n)}{\log^2(\alpha-n)}
+T, 
\end{align*}
where 
$$T=
\log\left(\frac{1+\frac{1}{g(n)}}{1+\frac{2}{g(n)}}\right)
-\frac{\log_2 (\alpha-n)}{\log (\alpha-n)} -
\frac{1}{\log \alpha}-
\frac{\log_2 \alpha}{\log^2 \alpha}+
\frac{3}{\log^2\alpha}-
\frac{3.1}{\log^2(\alpha-n)}
<0
$$ 
as $\frac{3}{\log^2\alpha}-
\frac{3.1}{\log^2(\alpha-n)}<0$.
$\qed$ 
\end{proof}
\section{Proofs of main results}
The following lemma shows that $G'(n)$ is a decreasing function for large $n$, which is 
crucial in the proof of Theorem 1.1.
\begin{lemma}
Let $\epsilon>0$ and  
$$g(n)=\frac{\log n +j(n)}{\log 2+\epsilon},$$ 
where $j(n)>0$ is a function that satisfies
$j(n)\to \infty$ and $nj'(n)\to 0$ as $n\to \infty$. Then 
$G'(n)\to -2\epsilon$.

\end{lemma}
\begin{proof}
We have $\left(\frac{n}{g(n)}\right)'=\frac{(\log 2+\epsilon)(\log n+j(n)-1-nj'(n))}
{(\log n+j(n))^2}$ and therefore 
$\left(\frac{n}{g(n)}\right)'\to 0$ as $n\to \infty$. 
By our assumption on $j(n)$ it follows that $\frac{j(n)}{\log n}\to 0$
which gives $ \left(\frac{n}{g(n)}\right)'\log(\alpha-n)\to \log 2+\epsilon$
(as $\frac{\log(\alpha-n)}{\log n}\to 1$). It is easy to see that  
$ \left(\frac{n}{g(n)}\right)'\log_2(\alpha-n)\to 0$. 
It follows that  
 $ \left(\frac{n}{g(n)}\right)'U'(\alpha-n)\to \log 2+\epsilon$ (see equation (1)). Lastly note that 
$f'(x)\to 0$ as $x\to \infty$. The result follows now on using all the above and the fact that 
$A\to \log 2$ (Lemma 2.3 part 1) in part 2 of Lemma 2.3. 
$\qed$

{\bf Proof of Theorem 1.1}
We will first show that there exists a positive integer $N$, such that 
$G(n)<0$ for $n>N$. We have  $G'(n)\to -2\epsilon$ by the lemma above, which means that  
if $0<\delta<2\epsilon$, then there exists an integer $M$, such that  for all $n>M$ 
we have $|G'(n)+2\epsilon|<\delta$, that is 
$$-2\epsilon-\delta<G'(n)<-2\epsilon+\delta, $$
for all $n>M$. 
Let $a$ and $b$ be two integers such that $M<a<b$. Then 
$G(b)-G(a)=\int_a^b G'(n) dn<(b-a) (-2\epsilon+\delta)<0$. If $a$ is fixed, it follows that 
 $G(b)<G(a)+(b-a)(-2\epsilon+\delta)<0$ for large $b$. 
 Therefore there exists a positive integer $N>M$, such that for all $n>N$, 
we have $G(n)=F(\alpha, n)<0$. 

 We may assume that $N>688383$ so that 
from Lemma 2.2, part 3  we have $F(s, n)>0$. Moreover, from the same lemma we have 
$F(x, n)$ is decreasing for $x\ge 2n$.
 As $s$ and $\alpha$ are both bigger than $2n$, we have $s<\alpha$  for $n>N$ and 
the result follows. 
$\qed$
\end{proof}
{\bf Proof of Corollary}

Let $\epsilon=\epsilon_1+\epsilon_2=0.5$. 
  We will first show that for $n>688383$ we have 
$G'(n)<0$. 

 Let $\epsilon_1=0.1$.  It is easy to verify that for 
$n>688383$ we have 
$$\frac{1+\log n}{\log n(\log n+\log_2 n-\log 2-\epsilon)}<
\frac{\epsilon_1}{\log 2+\epsilon}.$$ It follows that  
for all $n>688383$
\begin{equation}
\frac{ng(n)'}{g(n)^2}=
\frac{(\log 2+\epsilon)(1+\log n)}{\log n(\log n+\log_2n-\log 2-\epsilon)^2}
<\frac{\epsilon_1}{\log n+\log_2n-\log 2-\epsilon}.
\end{equation}

Next, we will show that  
 $A+f'(\alpha-n)-\log 2<\epsilon_2$. 

 Using Lemma 2.3, part 4 and Lemma 2.2 part 2,  we have  
\begin{equation}
 A+f'(\alpha-n)-\log 2< \log\left( \frac{\log(2.4n)}{\log n}\right) +
\frac{\log_2(2.4n)}{\log(2n)}+\frac{1.1}{\log n}+\frac{\log_2(1.4n)}{\log^2 n}.
\end{equation}
Observe that 
 for $n>36734$ 
\begin{equation}
\log\left( \frac{\log(2.4n)}{\log n}\right)
<\frac{\epsilon_2}{5}
\end{equation}
as $\log\left( \frac{\log(2.4n)}{\log n}\right)
<\frac{\epsilon_2}{5}$ holds if 
$\frac{\log(2.4n)}{\log n}<e^{\frac{\epsilon_2}{5}}$, 
that is if
$2.4n<n^{e^{\frac{\epsilon_2}{5}}}$. The above holds if  
$2.4<n^{e^{\frac{\epsilon_2}{5}-1}}$ or $n>36734$. 
 
Computation yields that for  $n>688383$  
\begin{equation}
\frac{\log_2(2.4n)}{\log(2n)}+\frac{1.1}{\log n}+\frac{\log_2(1.4n)}{\log^2 n}<\frac{4\epsilon_2}{5}.
\end{equation}

From equations (3)-(5) we have 
 $A+f'(\alpha-n)-\log 2<\epsilon_2$.  
From Lemma 2.3 part 3, 
$L'(\alpha-n)=U'(\alpha-n)-f'(\alpha-n)>\log (\alpha-n)+\log_2(\alpha-n)>\log n+\log_2 n$
and hence for $n>688383$ we have 
\begin{equation}
\frac{A+f'(\alpha-n)}{-A+U'(\alpha-n)-2f'(\alpha-n)}<
\frac{\log 2+\epsilon_2}{\log n+\log_2 n-\log 2-\epsilon_2}.
\end{equation}
As $\epsilon_1+\epsilon_2=\epsilon$,  
equations (2) and (6) give
\begin{equation}
\frac{A+f'(\alpha-n)}{-A+U'(\alpha-n)-2f'(\alpha-n)}+\frac{ng(n)'}{g(n)^2}<
\frac{\log 2+\epsilon_1+\epsilon_2}{\log n+\log_2n-\log 2-\epsilon}
=\frac{1}{g(n)}.
\end{equation}

From Lemma 2.3, part 2, noting that $\left(\frac{n}{g(n)}\right)'=\frac{1}{g(n)}-\frac{ng(n)'}{g(n)^2}$,  
we have 
$G'(n)<0$ for all $n>688383$. Also, $G(688383)<0$ and hence we conclude that 
$G(n)<0$ for $n>688383$. 

 From Lemma 2.2, part 3  we have $F(s, n)>0$ and $F(x, n)$ is decreasing for $x\ge 2n$.
 As $s$ and $\alpha$ are both bigger than $2n$,
it follows that 
$s<\alpha$ for $n>688383$.
That the result holds for $43<n\le 688383$ is a simple calculation.
$\qed$

\begin{rem}
Similar results for lower bounds for $R_n$ can be given using 
$G(x, n)=L(x)-2U(x-n+1)$ instead of $F(x, n)$. 
\end{rem}


\begin{thebibliography}{99}
\bibitem{Ax1}                 
{\textsc C. Axler},: \"{U}ber die Primzahl-Z\"{a}hlfunktion, die n-te Primzahl und verallgemeinerte Ramanujan– Primzahlen, Ph.D. thesis, 2013 (in German),
{\color{blue} \href{http://docserv.uni-duesseldorf.de/servlets/DocumentServlet?id=26247}
{http://docserv.uni-duesseldorf.de/servlets/DocumentServlet?id=26247}}
\bibitem{Ax}{\textsc C. Axler}, 
\textit{On generalized Ramanujan primes}, 
 Ramanujan journal, {\bf 39} (no. 1), (2016), 1--30. 

\bibitem{Du}{\textsc P. Dusart},
\textit{Estimates of some functions over primes without R.H.}, preprint (2010),
{\color{blue}\href{http://arxiv.org/abs/1002.0442}{http://arxiv.org/abs/1002.0442 }}

\bibitem{La}{\textsc S. Laishram},
\textit{On a conjecture on Ramanujan primes},  
  Int. J. Number Theory {\bf 6}, (2010), 1869--1873.  
\bibitem{Pak}{\textsc J. B. Paksoy},
\textit{Derived Ramanujan primes $R_n'$},
{\color{blue} \href{http:/arxiv.org/abs/1210.6991}{http:/arxiv.org/abs/1210.6991}} 
\bibitem{Sh}{\textsc 
V. Shevelev},
\textit{  Ramanujan and Labos primes, their generalizations and classifications of primes}, J.
Integer Seq. {\bf 15} (2012), Article 12.5.4
\bibitem{So}{\textsc J. Sondow},
\textit{Ramanujan primes and Bertrand's postulate}, 
  Amer. Math. Monthly {\bf 116}, (2009), 630--635.  
\bibitem{SNN}{\textsc J. Sondow, J. W. Nicholson, T. D. Noe},
\textit{Ramanujan primes: Bounds, Runs, Twins, and Gaps}, 
  Journal of integer sequences {\bf 14}, (2011), Article 11.6.2. 
\bibitem{As}{\textsc A. Srinivasan}, 
\textit{ An upper bound for Ramanujan primes}, 
Integers, {\bf 14}, no. A19 (2014).
\bibitem{SN}{\textsc A. Srinivasan and John Nicholson}, 
\textit{ An improved upper bound for Ramanujan primes}, 
Integers, {\bf 15}, no. A52 (2015).
\bibitem{YT}{\textsc S. Yang and A. Togbe} On the estimates of the upper and lower bounds
of Ramanujan primes,
 Ramanujan journal, {\bf 40} (no. 2), (2016), 245--255. 

\end{thebibliography}
\end{document}